\documentclass[11pt]{article}
\usepackage{amsthm, amsmath, amssymb, amsfonts, url, booktabs, tikz, setspace, fancyhdr, bm, mathrsfs}
\usepackage{hyperref}
\usepackage{geometry}
\usepackage{hyperref, enumerate}
\usepackage[shortlabels]{enumitem}
\usepackage[babel]{microtype}
\usepackage[english]{babel}
\usepackage[capitalise]{cleveref}
\usepackage{comment}
\usepackage{bbm}
\usepackage{csquotes}
\usepackage{mathabx}
\usepackage{tikz}
\usepackage{graphicx}
\usepackage{float}
\usepackage{amsmath}

\counterwithin{figure}{section}

\newtheorem{theorem}{Theorem}[section]
\newtheorem{prop}[theorem]{Proposition}
\newtheorem{lemma}[theorem]{Lemma}

\newtheorem{conj}[theorem]{Conjecture}
\newtheorem{claim}[theorem]{Claim}

\newtheorem{fact}[theorem]{Fact}

\theoremstyle{definition}

\newtheorem*{defn-non}{Definition}

\usepackage[linesnumbered, ruled]{algorithm2e}
\SetKwRepeat{Do}{do}{while}%

\newenvironment{poc}{\begin{proof}[Proof of the claim]}{\end{proof}}

\usepackage{todonotes}
\newcommand{\cA}{\mathcal{A}}

\newcommand*{\abs}[1]{\lvert#1\rvert}

\newcommand{\cB}{\mathcal{B}}
\newcommand{\cD}{\mathcal{D}}
\newcommand{\cE}{\mathcal{E}}
\newcommand{\cF}{\mathcal{F}}
\newcommand{\cG}{\mathcal{G}}
\newcommand{\cH}{\mathcal{H}}
\newcommand{\cK}{\mathcal{K}}
\newcommand{\cL}{\mathcal{L}}
\newcommand{\cM}{\mathcal{M}}
\newcommand{\cR}{\mathcal{R}}

\title{Largest Sperner families with restricted differences}
\author{
Gennian Ge\thanks{School of Mathematical Sciences, Capital Normal University, Beijing, China. Email: gnge@zju.edu.cn. Gennian Ge is supported by the National Key Research and Development Program of China under Grant 2025YFC3409900, the National Natural Science Foundation of China under Grant 12231014, and Beijing Scholars Program.}
\and
Zixiang Xu\thanks{School of Mathematical Sciences, Zhejiang University, Hangzhou, China. Email: zixiangxu@zju.edu.cn.}
\and
Xiaochen Zhao\thanks{School of Mathematical Sciences, Capital Normal University, Beijing, China. Email: 2250501013@cnu.edu.cn.}
}
\date{}

\begin{document}
\maketitle

\begin{abstract}
Let \(L\) be a fixed set of positive integers. A family \(\mathcal{F}\subseteq2^{[n]}\) is called \(L\)-differencing if \(\abs{A\setminus B}\in L\) for every ordered pair of distinct members \(A,B\in\mathcal{F}\). A longstanding conjecture of Frankl, proposed in 1985, asserts that every \(L\)-differencing family has size at most \(\binom{n}{|L|}\). We resolve this conjecture asymptotically for every fixed \(L\), and obtain the exact answer in the only case in which the conjectured bound could be tight.
\begin{enumerate}
\item If \(L\ne[s]\) and \(n\) is large, then every \(L\)-differencing family satisfies \(\abs{\mathcal{F}}\le\left(\frac{s}{s+1}+o_{L}(1)\right)\binom{n}{s}\).
\item If \(L=[s]\) and \(n\ge 2s-1\), then every \(L\)-differencing family satisfies \(\abs{\mathcal{F}}\le\binom{n}{s}\), with equality only for \(\binom{[n]}{s}\) and \(\binom{[n]}{n-s}\).
\end{enumerate}
The first result follows by reducing directed differences to restricted Hamming distances. For the exact result, we develop a new homogeneous polynomial method, which might be of independent interest.
\end{abstract}

\section{Introduction}
For a positive integer \(m\), write \([m]=\{1,\ldots,m\}\), and let \(\binom{[m]}{j}\) denote the family of all \(j\)-subsets of \([m]\). Sperner's theorem~\cite{S} asserts that every antichain in \(2^{[n]}\)
has size at most \(\binom{n}{\lfloor n/2\rfloor}\). It is one of the
starting points of extremal set theory, and many of its refinements ask
what happens when containment is forbidden together with additional
restrictions on pairwise set parameters. The literature around this theme
contains many restricted intersection, restricted distance, and restricted
difference problems~\cite{AK97,ABS91,Delsarte73,DGLOZ26,EKR61,FFP87,FranklWilson81,Katona64,Kleitman66,LL09,2021GC,RCW75,XuYip26}. In this paper the relevant
parameter is the directed difference \(\abs{A\setminus B}\).

Let \(L\) be a nonempty set of positive integers, and write \(s=\abs{L}\). A family
\(\mathcal{F}\subseteq 2^{[n]}\) is called \(L\)-differencing if
\(\abs{A\setminus B}\in L\) for every ordered pair of distinct members
\(A,B\in\cF\). Since \(0\notin L\), every such family is automatically
Sperner. Frankl~\cite{F85} proved, in a modular version of this problem,
that \(\abs{\cF}\le\sum_{i=0}^{s}\binom{n}{i}\), and asked whether the ordinary difference condition above always gives the stronger bound \(\binom{n}{s}\). The conjecture can be stated as follows.

\begin{conj}[\cite{F85}]\label{conj:frankl}
Let \(s\) be a positive integer, and let \(L\) be a set of \(s\) positive integers. There exists some \(n_{0}\) such that for every \(n\ge n_{0}\), every family \(\cF\subseteq2^{[n]}\) satisfying
\(\abs{A\setminus B}\in L\) for every ordered pair of distinct members
\(A,B\in\cF\) has size at most \(\binom{n}{s}\).
\end{conj}

Our results below settle this longstanding conjecture asymptotically for every fixed \(L\), and shows that a fixed positive proportion is lost unless \(L\) is the initial interval.

\begin{theorem}\label{thm:noninterval}
Let \(L\) be a fixed set of \(s\) positive integers with \(L\ne[s]\). If \(\cF\subseteq2^{[n]}\) is \(L\)-differencing, then, as \(n\to\infty\),
\[
\abs{\cF}\le\left(\frac{s}{s+1}+o_{L}(1)\right)\binom{n}{s}.
\]
\end{theorem}

Thus the interval case \(L=[k]=\{1,\ldots,k\}\) is the only case where the leading constant can be one. Indeed, the full level \(\binom{[n]}{k}\) is \([k]\)-differencing and has size \(\binom{n}{k}\). Liu and Liu~\cite{LL09} refined Frankl's modular estimate to \(\sum_{i=0}^{s}\binom{n-1}{i}\). Frankl~\cite{2017MJCNT} proved the sharp interval bound for \(n\ge k(k+2)\) through his theorem on antichains of diameter at most \(2k\). Xu and Yip~\cite{XuYip26} developed polynomial and \(p\)-adic methods for restricted-difference systems and obtained refined estimates in several settings. Gao, Liu, and Xu~\cite{2023StabilityComb} proved the sharp interval bound when \(n\) is sufficiently large in terms of \(k\), using a robust polynomial method.

We remove the large \(n\) assumption for the interval case and determine the exact extremal families.

\begin{theorem}\label{thm:main}
Let \(k\ge1\) and \(n\ge2k-1\), and let \(\cF\subseteq2^{[n]}\) be \([k]\)-differencing. Then \(\abs{\cF}\le\binom{n}{k}\). Moreover, equality holds if and only if \(\cF=\binom{[n]}{k}\) or \(\cF=\binom{[n]}{n-k}\).
\end{theorem}

The cases \(n=2k\) and \(n=2k-1\) follow from Sperner's theorem and its equality cases. The range in Theorem~\ref{thm:main} is best possible: if \(k\ge2\) and \(n=2k-2\), then \(\binom{[n]}{k-1}\) is \([k]\)-differencing, while \(\binom{2k-2}{k-1}>\binom{2k-2}{k}\).

\section{Proof of Theorem~\ref{thm:noninterval}}\label{sec:arbitrary}

For a family \(\cA\subseteq2^{[N]}\), let
\[
D(\cA)=\{\abs{A\triangle B}:A,B\in\cA,A\ne B\}
\]
be its set of nonzero Hamming distances. We shall use two bounds, one classical and one recent. The first is Delsarte's theorem~\cite{Delsarte73}.

\begin{theorem}\label{thm:delsarte}
If \(\cA\subseteq2^{[N]}\) and \(\abs{D(\cA)}=t\), then
\(
\abs{\cA}\le\sum_{i=0}^{t}\binom{N}{i}.
\)
\end{theorem}

The second input is a recent theorem of Dong, Gao, Liu, Ouyang, and Zhou~\cite[Theorem~1.4]{DGLOZ26}.

\begin{theorem}\label{thm:restricted-distances}
Let \(D\) be a fixed set of \(t\) positive integers. If \(D\ne\{2,4,\ldots,2t\}\), then every family \(\cA\subseteq2^{[N]}\) with \(D(\cA)=D\) satisfies, as \(N\to\infty\),
\[
\abs{\cA}\le\left(\frac{t}{t+1}+o_{D}(1)\right)\binom{N}{t}.
\]
\end{theorem}

The next lemma is the simple reduction from directed differences to ordinary Hamming distances. The point is that all sets can be put on one level by adding only a bounded number of new coordinates.

\begin{lemma}\label{lem:uniformization}
Let \(L=\{\ell_{1}<\cdots<\ell_{s}\}\), and let \(\varnothing\ne\cF\subseteq2^{[n]}\) be \(L\)-differencing. Then there is an integer \(q\), with \(0\le q\le\ell_{s}-\ell_{1}\), and an injective map \(\Phi:\cF\to2^{[n+q]}\) such that
\[
D(\Phi(\cF))\subseteq2L=\{2\ell:\ell\in L\}.
\]
\end{lemma}

\begin{proof}
Put \(r=\min_{A\in\cF}\abs{A}\), \(R=\max_{A\in\cF}\abs{A}\), and \(q=R-r\). If \(q>0\), choose \(A_{0},B_{0}\in\cF\) with \(\abs{A_{0}}=r\) and \(\abs{B_{0}}=R\). These two sets are distinct, so both directed differences belong to \(L\), and hence
\[
q=\abs{B_{0}\setminus A_{0}}-\abs{A_{0}\setminus B_{0}}\le\ell_{s}-\ell_{1},
\]
where the last inequality uses \(\abs{B_{0}\setminus A_{0}}\le\ell_{s}\) and \(\abs{A_{0}\setminus B_{0}}\ge\ell_{1}\).
If \(q=0\), the same bound is immediate. In either case, add new elements \(p_{1},\ldots,p_{q}\). For \(r\le h\le R\), put \(P_{h}=\{p_{1},\ldots,p_{R-h}\}\), and define
\[
\Phi(A)=A\cup P_{\abs{A}}.
\]
This map is injective, and every image has size \(R\).

Let \(A,B\in\cF\) be distinct, with \(a=\abs{A}\le b=\abs{B}\). Put \(x=\abs{A\setminus B}\) and \(y=\abs{B\setminus A}\). Since \(b-a=y-x\) and \(P_{a}\supseteq P_{b}\), we have
\[
\abs{\Phi(A)\setminus\Phi(B)}=x+(b-a)=y=\abs{\Phi(B)\setminus\Phi(A)}.
\]
Therefore \(\abs{\Phi(A)\triangle\Phi(B)}=2y\in2L\). This proves the lemma.
\end{proof}

\begin{proof}[Proof of Theorem~\ref{thm:noninterval}]
We may assume \(\abs{\cF}\ge2\). Apply Lemma~\ref{lem:uniformization}, put \(N=n+q\), and let \(D=D(\Phi(\cF))\). Write \(t=\abs{D}\). Since \(D\subseteq2L\), we have \(t\le s\).

If \(t\le s-1\), then Theorem~\ref{thm:delsarte} gives
\[
\abs{\cF}\le\sum_{i=0}^{t}\binom{N}{i}=O_{L}(n^{s-1})=o_{L}\left(\binom{n}{s}\right).
\]
It remains to consider \(t=s\). In this case \(D\subseteq2L\) and \(\abs{D}=\abs{2L}\), so \(D=2L\). Since \(L\ne[s]\), we have \(D\ne\{2,4,\ldots,2s\}\). Theorem~\ref{thm:restricted-distances} gives
\[
\abs{\cF}\le\left(\frac{s}{s+1}+o_{L}(1)\right)\binom{N}{s}.
\]
Finally, \(q\le\ell_{s}-\ell_{1}=O_{L}(1)\), and hence \(\binom{N}{s}=(1+O_{L}(n^{-1}))\binom{n}{s}\). This proves the theorem.
\end{proof}

\section{Proof of Theorem~\ref{thm:main}}
Since the cases \(n=2k\) and \(n=2k-1\) follow from Sperner's theorem, we always assume that \(n>2k\).
\subsection{Sharp upper bound}
 For \(A\subseteq[n]\), write \(\boldsymbol{1}_{A}\) for its characteristic vector. If \(\boldsymbol{x}=(x_{1},\ldots,x_{n})\) and \(S\subseteq[n]\), write \(\boldsymbol{x}_{S}=\prod_{i\in S}x_{i}\), with \(\boldsymbol{x}_{\varnothing}=1\). For \(z\in\mathbb{R}\) and \(j\ge0\), we use the convention \(\binom{z}{j}=\frac{z(z-1)\cdots(z-j+1)}{j!}\). We also write \((z)_{d}=z(z-1)\cdots(z-d+1)\), with \((z)_{0}=1\), and let \([z^{d}]f(z)\) denote the coefficient of \(z^{d}\) in a polynomial or formal power series \(f(z)\).

We start by listing the three basic tools used in the proof. If \(G\) is a bipartite graph with parts \(X\) and \(Y\), a matching from \(X\) into \(Y\) means a set of disjoint edges covering every vertex of \(X\). For \(U\subseteq X\), let \(N(U)\) be the set of vertices in \(Y\) adjacent to at least one vertex of \(U\).

\begin{theorem}[\cite{Hall35}]\label{thm:hall}
Let \(G\) be a bipartite graph with parts \(X\) and \(Y\). There is a matching from \(X\) into \(Y\) if and only if \(\abs{N(U)}\ge\abs{U}\) for every \(U\subseteq X\).
\end{theorem}
The next proposition is a standard linear algebra step in the polynomial method; see, for example, Babai and Frankl~\cite{BabaiFrankl92}.
\begin{prop}\label{prop:triangular-criterion}
Let \(f_{1},\ldots,f_{m}\) be functions on a common set \(\cD\), and let \(\boldsymbol{v}_{1},\ldots,\boldsymbol{v}_{m}\in\cD\). If \(f_{i}(\boldsymbol{v}_{i})\ne0\) for every \(i\) and \(f_{i}(\boldsymbol{v}_{j})=0\) whenever \(i<j\), then \(f_{1},\ldots,f_{m}\) are linearly independent.
\end{prop}

\begin{proof}[Proof of Lemma~\ref{prop:triangular-criterion}]
Suppose that \(\sum_{i=1}^{m}c_{i}f_{i}=0\), and let \(j\) be the largest index for which \(c_{j}\ne0\). Evaluating at \(\boldsymbol{v}_{j}\), all terms with index smaller than \(j\) vanish by assumption, while all terms with larger index have zero coefficient. We obtain \(c_{j}f_{j}(\boldsymbol{v}_{j})=0\), a contradiction.
\end{proof}

We will also take advantage of the following simple formula.
\begin{fact}\label{fact:finite-difference}
If \(b\ge1\) and \(f\) is a polynomial of degree less than \(b\), then \(\sum_{t=0}^{b}(-1)^{t}\binom{b}{t}f(t)=0\).
\end{fact}

\begin{proof}[Proof of Fact~\ref{fact:finite-difference}]
Every polynomial of degree less than \(b\) is a linear combination of \(\binom{t}{d}\) for \(0\le d<b\). It is therefore enough to use \(\binom{b}{t}\binom{t}{d}=\binom{b}{d}\binom{b-d}{t-d}\) and compute
\[
\sum_{t=0}^{b}(-1)^{t}\binom{b}{t}\binom{t}{d}=(-1)^{d}\binom{b}{d}\sum_{u=0}^{b-d}(-1)^{u}\binom{b-d}{u}=0.
\]
\end{proof}

Now we formally turn to the proof of Theorem~\ref{thm:main}. We first remove a small technical issue: the polynomial construction below works cleanly once every set has size at least \(k\), and Hall's theorem lets us enforce this without changing the size of the family.

\begin{prop}\label{prop:lifting}
Let \(n>2k\), and let \(\cF\subseteq2^{[n]}\) be \([k]\)-differencing. There is a \([k]\)-differencing family \(\cF'\subseteq2^{[n]}\) such that \(\abs{\cF'}=\abs{\cF}\) and \(\abs{A}\ge k\) for all \(A\in\cF'\).
\end{prop}
\begin{proof}[Proof of Proposition~\ref{prop:lifting}]
Suppose the minimum size of a member of \(\cF\) is \(r<k\), and put \(\cF_{r}=\cF\cap\binom{[n]}{r}\). Consider the bipartite inclusion graph between \(\cF_{r}\) and \(\binom{[n]}{r+1}\). If \(\cG\subseteq\cF_{r}\) and \(N(\cG)\) is its upper neighborhood, then counting edges gives
\[
(n-r)\abs{\cG}\le(r+1)\abs{N(\cG)}.
\]
Since \(r<k\) and \(n>2k\), we have \(n-r>r+1\), and hence \(\abs{N(\cG)}\ge\abs{\cG}\). By Theorem~\ref{thm:hall}, there is an injection \(\phi:\cF_{r}\to\binom{[n]}{r+1}\) with \(A\subseteq\phi(A)\) for every \(A\in\cF_{r}\).

Replace each \(A\in\cF_{r}\) by \(A^{+}=\phi(A)\), and leave all other members unchanged. The injection \(\phi\) makes the lifted sets distinct. Moreover, no lifted set \(A^{+}\) is an unchanged member of \(\cF\), since \(A\subset A^{+}\) would then contradict that \(\cF\) is an antichain.

The new family is again an antichain. If \(B\) is unchanged, then \(A^{+}\subseteq B\) would imply \(A\subset B\). In the other direction, \(B\subseteq A^{+}\) and the minimality of \(r\) give \(r+1\le\abs{B}\le\abs{A^{+}}=r+1\), so \(B=A^{+}\), again contradicting \(A\subset B\). Two lifted sets are distinct and have the same size, so they are also incomparable.

It remains to check the difference condition. All directed differences are positive because the new family is an antichain. If \(B\) is unchanged, then \(\abs{B\setminus A^{+}}\le\abs{B\setminus A}\le k\) and \(\abs{A^{+}\setminus B}\le r+1\le k\). For two lifted sets, both directed differences are at most \(r+1\le k\). Thus the new family is \([k]\)-differencing. Repeating this step until the minimum size reaches \(k\) proves the proposition.
\end{proof}

\begin{claim}\label{clm:lifting-trajectory}
Follow a set through the successive replacements in one complete run of the lifting procedure. If the set is replaced at least once, then its final replacement has size \(k\).
\end{claim}

\begin{poc}
Suppose the set is first replaced when the current minimum size is \(r<k\). Its replacement has size \(r+1\). Since every set on level \(r\) is replaced at that step, the new minimum size is \(r+1\), so this replacement is lifted again at the next step unless \(r+1=k\). Continuing in this way, its successive replacements have sizes \(r+1,r+2,\ldots,k\).
\end{poc}

Replacing \(\cF\) by the family supplied by Proposition~\ref{prop:lifting} does not change its size. From now on, we may therefore assume that \(\abs{A}\ge k\) for every \(A\in\cF\). We also record a simple consequence of the \([k]\)-differencing property.
\begin{claim}\label{clm:basic-differences}
If \(A,B\in\cF\) are distinct, \(\abs{A}\le\abs{B}\), \(a=\abs{A\setminus B}\), and \(b=\abs{B\setminus A}\), then \(1\le a\le b\le k\).
\end{claim}

\begin{poc}
Since \(\cF\) is an antichain, \(a,b\ge1\). Since it is \([k]\)-differencing, \(a,b\le k\). Finally,
\[
\abs{B}-\abs{A}=\abs{B\setminus A}-\abs{A\setminus B}=b-a\ge0,
\]
so \(a\le b\).
\end{poc}

We now construct the polynomials. Let \(\cH_{k}=\operatorname{span}_{\mathbb{R}}\{\boldsymbol{x}_{T}:\abs{T}=k\}\), so \(\dim\cH_{k}=\binom{n}{k}\). The usual polynomial attached to \(A\in\cF\) is
\[
P_{A}(\boldsymbol{x})=\prod_{\ell=1}^{k}\left(\abs{A}-\sum_{i\in A}x_{i}-\ell\right).
\]
Let \(P_{A}^{\mathrm{ml}}\) be the multilinear reduction of \(P_{A}\), obtained by replacing every positive power of \(x_{i}\) by \(x_{i}\). Thus \(P_{A}^{\mathrm{ml}}\) is multilinear and agrees with \(P_{A}\) on \(\{0,1\}^{n}\). If \(B\in\cF\setminus\{A\}\), then the factor indexed by \(\ell=\abs{A\setminus B}\) shows that \(P_{A}(\boldsymbol{1}_{B})=0\), while \(P_{A}(\boldsymbol{1}_{A})=(-1)^{k}k!\). In general, \(P_{A}^{\mathrm{ml}}\) may contain terms of several degrees between \(0\) and \(k\); when \(\abs{A}=k\), it is already homogeneous. The key point is that we can replace it by a polynomial using only degree-\(k\) monomials.

For \(0\le j\le k\), put \(e_{j}(A;\boldsymbol{x})=\sum_{S\subseteq A,\abs{S}=j}\boldsymbol{x}_{S}\).

\begin{claim}\label{clm:PA-expansion}
Let \(A\subseteq[n]\), and let \(r=\abs{A}\ge k\). Then
\[
P_{A}^{\mathrm{ml}}(\boldsymbol{x})=k!\sum_{j=0}^{k}(-1)^{j}\binom{r-j-1}{k-j}e_{j}(A;\boldsymbol{x}).
\]
\end{claim}

\begin{poc}
Let \(R(\boldsymbol{x})\) denote the right side of the claimed identity. Both \(P_{A}^{\mathrm{ml}}\) and \(R\) are multilinear polynomials involving only the coordinates in \(A\). Hence it is enough to compare them at every \(\boldsymbol{u}\in\{0,1\}^{n}\).

Fix \(\boldsymbol{u}\in\{0,1\}^{n}\), and let \(t=\abs{\{i\in A:u_{i}=1\}}\). Since \(P_{A}^{\mathrm{ml}}\) and \(P_{A}\) agree on Boolean vectors, we have
\[
\frac{P_{A}^{\mathrm{ml}}(\boldsymbol{u})}{k!}=\frac{(r-t-1)(r-t-2)\cdots(r-t-k)}{k!}=\binom{r-t-1}{k}.
\]
Now consider \(R(\boldsymbol{u})\). In the sum defining \(e_{j}(A;\boldsymbol{u})\), the monomial \(\boldsymbol{u}_{S}\) equals one exactly when all elements of \(S\) lie among the \(t\) positions of \(A\) where \(u_{i}=1\). There are \(\binom{t}{j}\) such choices of \(S\), so \(e_{j}(A;\boldsymbol{u})=\binom{t}{j}\). It follows that
\[
\frac{R(\boldsymbol{u})}{k!}=\sum_{j=0}^{k}(-1)^{j}\binom{t}{j}\binom{r-j-1}{k-j}.
\]
To evaluate this sum, read it as the coefficient of \(z^{k}\). Terms with \(j>k\) cannot contribute to this coefficient, so the binomial theorem gives
\[
\sum_{j=0}^{k}(-1)^{j}\binom{t}{j}\binom{r-j-1}{k-j}=[z^{k}](1+z)^{r-1}\sum_{j=0}^{t}\binom{t}{j}\left(\frac{-z}{1+z}\right)^{j}=[z^{k}](1+z)^{r-t-1}=\binom{r-t-1}{k}.
\]
When \(r-t-1=-1\), the coefficient is read in the formal power series expansion, in agreement with our convention for generalized binomial coefficients. Thus \(P_{A}^{\mathrm{ml}}(\boldsymbol{u})=R(\boldsymbol{u})\) for every Boolean vector \(\boldsymbol{u}\), proving the claim.
\end{poc}

For each \(A\subseteq[n]\) with \(r=\abs{A}\ge k\), define the degree-\(k\) polynomial \(H_{A}\in\cH_{k}\) by
\begin{equation}\label{eq:HA-definition}
H_{A}(\boldsymbol{x})=\begin{cases}\boldsymbol{x}_{A},&r=k,\\\displaystyle\sum_{T\in\binom{[n]}{k}}(-1)^{\abs{T\cap A}}\abs{T\cap A}!(r-\abs{T\cap A}-1)!\boldsymbol{x}_{T},&r>k.\end{cases}
\end{equation}
This is the object we will use in place of \(P_{A}\). The main point of the construction is the vanishing across different levels, which will be proved in Claim~\ref{clm:triangularity}. The next proposition shows that \(H_{A}\) is a faithful degree-\(k\) replacement for \(P_{A}\): on the level containing \(A\), the two polynomials differ only by a nonzero constant factor.

\begin{prop}\label{prop:homogeneous}
Let \(A\subseteq[n]\) and \(r=\abs{A}\ge k\). For every \(\boldsymbol{x}\in\Omega_{r}=\{\boldsymbol{x}\in\{0,1\}^{n}:\sum_{i=1}^{n}x_{i}=r\}\), we have \(H_{A}(\boldsymbol{x})=c_{r}P_{A}(\boldsymbol{x})\), where
\[
c_{r}=\begin{cases}\displaystyle\frac{(-1)^{k}}{k!},&r=k,\\\displaystyle\frac{r!}{k!(r-k)},&r>k.\end{cases}
\]
\end{prop}
\begin{proof}[Proof of Proposition~\ref{prop:homogeneous}]
The case \(r=k\) is immediate. On \(\Omega_{k}\), the monomial \(\boldsymbol{x}_{A}\) is one at \(\boldsymbol{1}_{A}\) and zero at every other Boolean vector, and a direct evaluation gives \(\boldsymbol{x}_{A}=\frac{(-1)^{k}}{k!}P_{A}(\boldsymbol{x})\).

Assume now that \(r>k\). We first introduce a normalized degree-\(k\) lifting operation adapted to the \(r\)-th level. For a monomial \(\boldsymbol{x}_{S}\) with \(|S|\le k\), we define
\begin{equation}\label{eq:raising}
U_{r,k}(\boldsymbol{x}_{S})=\frac{1}{\binom{r-\abs{S}}{k-\abs{S}}}\sum_{\substack{T\supseteq S\\\abs{T}=k}}\boldsymbol{x}_{T}.
\end{equation}
If \(f=\sum_{\abs{S}\le k}c_{S}\boldsymbol{x}_{S}\), define \(U_{r,k}(f)\) by making this replacement in every monomial:
\[
U_{r,k}(f)=\sum_{\abs{S}\le k}c_{S}U_{r,k}(\boldsymbol{x}_{S}).
\]
Now take \(B\in\binom{[n]}{r}\). The two functions \(U_{r,k}(\boldsymbol{x}_{S})\) and \(\boldsymbol{x}_{S}\) have the same value at \(\boldsymbol{1}_{B}\). If \(S\not\subseteq B\), both values are zero. If \(S\subseteq B\), exactly \(\binom{r-\abs{S}}{k-\abs{S}}\) summands in \eqref{eq:raising} are equal to one, so the normalized sum is one. Making this replacement term by term gives \(U_{r,k}(f)=f\) on \(\Omega_{r}\) for every multilinear polynomial \(f\) of degree at most \(k\).

We now apply this to \(f=P_{A}^{\mathrm{ml}}\). To compare \(U_{r,k}(P_{A}^{\mathrm{ml}})\) with \(H_{A}\), fix \(T\in\binom{[n]}{k}\), and set \(q=\abs{T\cap A}\). A monomial \(\boldsymbol{x}_{S}\) from \(e_{j}(A;\boldsymbol{x})\) contributes to the coefficient of \(\boldsymbol{x}_{T}\) only when \(S\subseteq T\cap A\). There are \(\binom{q}{j}\) such sets \(S\), so \eqref{eq:raising} and Claim~\ref{clm:PA-expansion} show that the coefficient of \(\boldsymbol{x}_{T}\) in \(U_{r,k}(P_{A}^{\mathrm{ml}})\) is
\[
k!\sum_{j=0}^{q}(-1)^{j}\binom{q}{j}\frac{\binom{r-j-1}{k-j}}{\binom{r-j}{k-j}}=k!(r-k)\sum_{j=0}^{q}(-1)^{j}\binom{q}{j}\frac{1}{r-j},
\]
where we used \(\frac{\binom{r-j-1}{k-j}}{\binom{r-j}{k-j}}=\frac{r-k}{r-j}\).
We use the elementary identity
\begin{equation}\label{eq:beta-sum}
\sum_{j=0}^{q}(-1)^{j}\binom{q}{j}\frac{1}{r-j}=(-1)^{q}\frac{q!(r-q-1)!}{r!}.
\end{equation}
Indeed, expanding \((1-y)^{q}\) and then replacing the index \(h\) by \(q-j\) gives
\[
\int_{0}^{1}y^{r-q-1}(1-y)^{q}\mathrm{d}y=\sum_{h=0}^{q}(-1)^{h}\binom{q}{h}\frac{1}{r-q+h}=(-1)^{q}\sum_{j=0}^{q}(-1)^{j}\binom{q}{j}\frac{1}{r-j}.
\]
The integral equals \(\frac{q!(r-q-1)!}{r!}\), which proves \eqref{eq:beta-sum}. Therefore the coefficient of \(\boldsymbol{x}_{T}\) in \(U_{r,k}(P_{A}^{\mathrm{ml}})\) is \(k!(r-k)(-1)^{q}\frac{q!(r-q-1)!}{r!}\). This is \(\frac{k!(r-k)}{r!}\) times the coefficient of \(\boldsymbol{x}_{T}\) in \eqref{eq:HA-definition}. Hence
\[
H_{A}=\frac{r!}{k!(r-k)}U_{r,k}(P_{A}^{\mathrm{ml}}).
\]
Since \(U_{r,k}(P_{A}^{\mathrm{ml}})=P_{A}^{\mathrm{ml}}=P_{A}\) on \(\Omega_{r}\), the proposition follows.
\end{proof}
The most important point is that these polynomials still vanish in the order we need, even when they are evaluated on sets of different sizes.
\begin{claim}\label{clm:triangularity}
Let \(A,B\in\cF\) and \(\abs{B}\ge\abs{A}\). Then
\[
H_{A}(\boldsymbol{1}_{B})=\begin{cases}0,&A\ne B,\\1,&A=B\text{ and }\abs{A}=k,\\\displaystyle(-1)^{k}\frac{\abs{A}!}{\abs{A}-k},&A=B\text{ and }\abs{A}>k.\end{cases}
\]
\end{claim}
\begin{poc}
Put \(r=\abs{A}\). If \(r=k\), then \(H_{A}=\boldsymbol{x}_{A}\). Since \(\cF\) is an antichain and \(\abs{B}\ge\abs{A}\), the value is zero off the diagonal and one on the diagonal.

Assume \(r>k\) and \(A\ne B\). Let \(a=\abs{A\setminus B}\) and \(b=\abs{B\setminus A}\). By Claim~\ref{clm:basic-differences}, \(1\le a\le b\le k\). Evaluating \eqref{eq:HA-definition} at \(\boldsymbol{1}_{B}\) and grouping surviving monomials by \(q=\abs{T\cap A}\) gives
\begin{equation}\label{eq:eval-q}
H_{A}(\boldsymbol{1}_{B})=\sum_{q=0}^{k}(-1)^{q}q!(r-q-1)!\binom{r-a}{q}\binom{b}{k-q}.
\end{equation}
For \(q\le k<r\) we have
\[
q!(r-q-1)!\binom{r-a}{q}=(r-a)!(r-q-1)_{a-1},
\]
when \(q\le r-a\). If \(q>r-a\), the binomial coefficient on the left is zero, while \(r-q-1<a-1\), so the falling factorial on the right contains a zero factor. Thus the same identity holds in this case as well. Since \(b\le k\), substituting \(q=k-b+t\) in \eqref{eq:eval-q} yields
\[
H_{A}(\boldsymbol{1}_{B})=(-1)^{k-b}(r-a)!\sum_{t=0}^{b}(-1)^{t}\binom{b}{t}(r-k+b-t-1)_{a-1}.
\]
The last falling factorial is a polynomial in \(t\) of degree \(a-1\). Since \(a\le b\) and in fact \(a-1<b\), Fact~\ref{fact:finite-difference} applies and gives zero.

Finally, at \(\boldsymbol{1}_{A}\) only monomials \(\boldsymbol{x}_{T}\) with \(T\in\binom{A}{k}\) contribute. Therefore
\[
H_{A}(\boldsymbol{1}_{A})=(-1)^{k}\binom{r}{k}k!(r-k-1)!=\frac{(-1)^{k}r!}{r-k}.
\]
This proves the claim.
\end{poc}
List \(\cF=\{A_{1},\ldots,A_{m}\}\) with \(\abs{A_{1}}\le\cdots\le\abs{A_{m}}\). By Claim~\ref{clm:triangularity}, the matrix \(\bigl(H_{A_{i}}(\boldsymbol{1}_{A_{j}})\bigr)_{1\le i,j\le m}\) is triangular with nonzero diagonal. By Proposition~\ref{prop:triangular-criterion}, \(H_{A_{1}},\ldots,H_{A_{m}}\) are linearly independent in \(\cH_{k}\). Hence
\[
\abs{\cF}=m\le\dim\cH_{k}=\binom{n}{k}.
\]
This proves the upper bound in Theorem~\ref{thm:main} when \(n>2k\).

\subsection{Extremal configurations}
The two full levels \(\binom{[n]}{k}\) and \(\binom{[n]}{n-k}\) attain equality. For the second family, if \(A\ne B\) are \((n-k)\)-sets, then \(A\setminus B\subseteq B^{c}\), so \(1\le\abs{A\setminus B}\le k\). It remains to show that these are the only extremal families. The equality proof has three steps. First, equality in the polynomial bound forces a balanced-shadow identity for every initial part of the family inside the middle band \(k\le\abs{A}\le n-k\). Second, Katona's equality case for intersecting shadows turns this identity into a rigidity statement: a middle-band equality case is one of the two full levels. Finally, we show that the lifting reductions used to reach the middle band must have been trivial.

For a family \(\cA\subseteq2^{[n]}\), write
\[
\partial_{k}\cA=\left\{T\in\binom{[n]}{k}:T\subseteq A\textup{ for some }A\in\cA\right\}
\]
for its \(k\)-shadow. We shall use the following form of Katona's intersecting shadow theorem~\cite{Katona64}; the strict equality statement is recorded explicitly in~\cite[Theorem~1.3]{FranklKatona21}.

\begin{lemma}[Katona's intersecting shadow theorem]\label{lem:katona-shadow}
Let \(m>k\), let \(n\ge m+k\), and let \(\varnothing\ne\cA\subseteq\binom{[n]}{m}\) satisfy \(\abs{A\cap A'}\ge m-k\) for all \(A,A'\in\cA\). Then \(\abs{\partial_{k}\cA}\ge\abs{\cA}\). Equality holds if and only if there is a set \(Q\subseteq[n]\) of size \(m+k\) such that \(\cA=\binom{Q}{m}\).
\end{lemma}

\begin{prop}\label{prop:middle-band-equality}
Assume \(n>2k\). Let \(\cG\subseteq2^{[n]}\) be \([k]\)-differencing with \(\abs{\cG}=\binom{n}{k}\), and suppose that \(k\le\abs{A}\le n-k\) for every \(A\in\cG\). Then \(\cG=\binom{[n]}{k}\) or \(\cG=\binom{[n]}{n-k}\).
\end{prop}
\begin{proof}[Proof of Proposition~\ref{prop:middle-band-equality}]
Put \(d=\binom{n}{k}\) and \(\cK=\binom{[n]}{k}\). We first obtain two bijections between \(\cK\) and \(\cG\). List \(\cG=\{A_{1},\ldots,A_{d}\}\) so that \(\abs{A_{1}}\le\cdots\le\abs{A_{d}}\), and list \(\cK=\{T_{1},\ldots,T_{d}\}\) in any order. By \eqref{eq:HA-definition}, write
\[
H_{A_{i}}(\boldsymbol{x})=\sum_{j=1}^{d}c_{i,j}\boldsymbol{x}_{T_{j}}.
\]
Define the \(d\times d\) matrices \(C=(c_{i,j})\) and \(W=(w_{j,\ell})\), where \(w_{j,\ell}=1\) if \(T_{j}\subseteq A_{\ell}\), and \(w_{j,\ell}=0\) otherwise. Then
\[
(CW)_{i,\ell}=H_{A_{i}}(\boldsymbol{1}_{A_{\ell}}).
\]
The rows and columns of \(CW\) are therefore indexed by the same ordered list \(A_{1},\ldots,A_{d}\). Claim~\ref{clm:triangularity} shows that \(CW\) is triangular with nonzero diagonal. Hence \(CW\) is invertible, and since \(C\) and \(W\) are square matrices, \(W\) is invertible.

Expanding \(\det W\), at least one product in the determinant expansion is nonzero. All entries of \(W\) are \(0\) or \(1\), so this product chooses one nonzero entry from each row and each column. Therefore there is a bijection \(\alpha:\cK\to\cG\) such that
\(
T\subseteq\alpha(T)
\)
for every \(T\in\cK\).

We repeat the same argument for \(\cG^{c}=\{B^{c}:B\in\cG\}\). This family is still \([k]\)-differencing, since \(\abs{A^{c}\setminus B^{c}}=\abs{B\setminus A}\), and it is still in the same middle band. Translating \(T\subseteq B^{c}\) back to \(B\), we get an invertible matrix
\[
D=(d_{T,B})_{T\in\cK,B\in\cG},
\]
where \(d_{T,B}=1\) if \(T\cap B=\varnothing\), and \(d_{T,B}=0\) otherwise. Hence there is a bijection \(\beta:\cK\to\cG\) such that
\(
T\cap\beta(T)=\varnothing
\)
for every \(T\in\cK\).

Fix \(T\in\cK\). Since \(T\subseteq\alpha(T)\) and \(T\cap\beta(T)=\varnothing\), we have \(T\subseteq\alpha(T)\setminus\beta(T)\). Also \(\alpha(T)\ne\beta(T)\), because \(k\ge1\). The difference bound gives \(\abs{\alpha(T)\setminus\beta(T)}\le k\). This set already contains the \(k\)-set \(T\), so
\(
\alpha(T)\setminus\beta(T)=T.
\)
Using the difference bound in the other direction,
\[
\abs{\alpha(T)}-\abs{\beta(T)}=\abs{\alpha(T)\setminus\beta(T)}-\abs{\beta(T)\setminus\alpha(T)}\ge k-k=0.
\]
Thus \(\abs{\alpha(T)}\ge\abs{\beta(T)}\). Let \(\pi=\beta\circ\alpha^{-1}\), which is a permutation of \(\cG\). The last inequality says that \(\abs{A}\ge\abs{\pi(A)}\) for every \(A\in\cG\). Summing these inequalities around any cycle of \(\pi\) gives equality of the two sums, so every inequality on the cycle is equality. Hence
\(
\abs{\alpha(T)}=\abs{\beta(T)}
\)
for every \(T\in\cK\). Denote this common value by \(\rho(T)\).
\begin{claim}\label{clm:balanced-shadow}
For every integer \(s\),
\(
\partial_{k}\cG_{\le s}=\{T\in\cK:\rho(T)\le s\},
\)
where \(\cG_{\le s}=\{A\in\cG:\abs{A}\le s\}\).
\end{claim}
\begin{poc}
If \(\rho(T)\le s\), then \(T\subseteq\alpha(T)\) and \(\abs{\alpha(T)}=\rho(T)\le s\). Hence \(T\in\partial_{k}\cG_{\le s}\).

Conversely, suppose \(T\subseteq A\) for some \(A\in\cG_{\le s}\), and put \(B=\beta(T)\). Then \(T\subseteq A\setminus B\). The sets \(A\) and \(B\) are distinct, since \(T\subseteq A\) but \(T\cap B=\varnothing\). The difference bound gives \(\abs{A\setminus B}\le k=\abs{T}\), and therefore \(A\setminus B=T\). Also \(\abs{B\setminus A}\le k\), so
\[
\abs{B}=\abs{A\cap B}+\abs{B\setminus A}=\abs{A}-\abs{A\setminus B}+\abs{B\setminus A}\le\abs{A}-k+k\le s.
\]
Thus \(\rho(T)=\abs{\beta(T)}=\abs{B}\le s\).
\end{poc}
Since \(\alpha\) is a bijection and \(\abs{\alpha(T)}=\rho(T)\), Claim~\ref{clm:balanced-shadow} gives
\(
\abs{\partial_{k}\cG_{\le s}}=\abs{\cG_{\le s}}
\)
for every \(s\). We now use this identity at the lowest levels of \(\cG\).

Let \(m=\min\{\abs{A}:A\in\cG\}\). Suppose first that \(m>k\), and put \(\cA=\cG\cap\binom{[n]}{m}\). Since \(m\) is the lowest level, \(\cG_{\le m}=\cA\), and the identity above gives \(\abs{\partial_{k}\cA}=\abs{\cA}\). For \(A,A'\in\cA\), the difference bound gives \(\abs{A\setminus A'}\le k\), hence \(\abs{A\cap A'}\ge m-k\). Also \(m\le n-k\), so \(n\ge m+k\). Lemma~\ref{lem:katona-shadow} applies, and we get a set \(Q\subseteq[n]\) of size \(m+k\) such that \(\cA=\binom{Q}{m}\).

We now show that no higher level can occur.
\begin{claim}\label{claim:large}
If \(m>k\), then \(\cG\) has no member of size larger than \(m\). In particular, \(\cG=\binom{Q}{m}\).
\end{claim}

\begin{poc}
Suppose for a contradiction that \(X\in\cG\) has size \(s>m\), and put \(q=\abs{X\cap Q}\). We compare \(X\) with suitable members of \(\binom{Q}{m}=\cA\).

First suppose \(q\le k\). Then \(\abs{Q\setminus X}=m+k-q\ge m\), so we can choose \(A\in\binom{Q}{m}\) disjoint from \(X\). This gives \(\abs{X\setminus A}=s>k\), contradicting the difference bound.

It remains to consider \(q>k\). Choose \(A\in\binom{Q}{m}\) by taking all \(m+k-q\) points of \(Q\setminus X\) and then \(q-k\) points of \(Q\cap X\). This choice is possible because \((m+k-q)+(q-k)=m\). It leaves exactly \(k\) points of \(Q\cap X\) outside \(A\), and all \(s-q\) points of \(X\setminus Q\) are also outside \(A\). Hence \(\abs{X\setminus A}=s-q+k\). The difference bound gives \(s-q+k\le k\), so \(s\le q\). Since \(q=\abs{X\cap Q}\le\abs{X}=s\), we get \(q=s\), and hence \(X\subseteq Q\). But \(s>m\), so \(X\) contains an \(m\)-subset \(A'\) belonging to \(\binom{Q}{m}=\cA\), contradicting the antichain property. Therefore no such \(X\) exists, and \(\cG=\cA=\binom{Q}{m}\).
\end{poc}
By Claim~\ref{claim:large}, \(\abs{\cG}=\binom{m+k}{k}\). Since also \(\abs{\cG}=\binom{n}{k}\), while \(m+k\le n\), the strict increase of \(u\mapsto\binom{u}{k}\) for \(u\ge k\) gives \(m+k=n\). Therefore \(Q=[n]\), \(m=n-k\), and \(\cG=\binom{[n]}{n-k}\).

It remains to handle \(m=k\). The next claim says that no level above \(k\) can occur.

\begin{claim}\label{claim:equalSize}
If \(m=k\), then \(\cG=\binom{[n]}{k}\).
\end{claim}

\begin{poc}
If \(\cG\) has no member above level \(k\), then \(\cG\subseteq\binom{[n]}{k}\), and \(\abs{\cG}=\binom{n}{k}\) gives \(\cG=\binom{[n]}{k}\). Thus suppose that a larger level occurs. Let \(r>k\) be the next nonempty level, and put \(\cB=\cG\cap\binom{[n]}{r}\).

We first isolate what the balanced-shadow identity says at level \(r\). The two families \(\cG\cap\binom{[n]}{k}\) and \(\partial_{k}\cB\) are disjoint: otherwise a \(k\)-set in \(\cG\) would be contained in a member of \(\cB\), contradicting the antichain property. Since there are no nonempty levels between \(k\) and \(r\), we have
\(
\cG_{\le r}=\left(\cG\cap\binom{[n]}{k}\right)\sqcup\cB
\)
and
\(
\partial_{k}\cG_{\le r}=\left(\cG\cap\binom{[n]}{k}\right)\sqcup\partial_{k}\cB.
\)
Therefore \(\abs{\partial_{k}\cG_{\le r}}=\abs{\cG_{\le r}}\) gives \(\abs{\partial_{k}\cB}=\abs{\cB}\) after cancelling the \(k\)-sets already in \(\cG\).

Now Lemma~\ref{lem:katona-shadow} applies to \(\cB\). Indeed, for \(X_{1},X_{2}\in\cB\), the difference bound gives \(\abs{X_{1}\cap X_{2}}\ge r-k\), and the middle-band condition gives \(r\le n-k\), or \(n\ge r+k\). Hence \(\cB=\binom{Q}{r}\) for some \(Q\subseteq[n]\) of size \(r+k\).

Choose \(K\in\cG\cap\binom{[n]}{k}\), which exists because \(m=k\). If \(K\subseteq Q\), then some \(Y\in\binom{Q}{r}\) contains \(K\), contradicting the antichain property. If \(K\nsubseteq Q\), then \(\abs{K\cap Q}\le k-1\), so
\[
\abs{Q\setminus K}\ge(r+k)-(k-1)=r+1.
\]
Choose \(Y\in\binom{Q}{r}\) disjoint from \(K\). Then \(\abs{Y\setminus K}=r>k\), contradicting the difference bound. Both cases are impossible, so no larger level occurs. Hence \(\cG=\binom{[n]}{k}\).
\end{poc}
Since \(m\ge k\), Claims~\ref{claim:large} and~\ref{claim:equalSize} prove the proposition.
\end{proof}

\begin{lemma}\label{lem:no-lift-to-full-k}
Assume \(n>2k\). Let \(\cE\subseteq2^{[n]}\) be \([k]\)-differencing, and let \(\cE^{\uparrow}\) be obtained from \(\cE\) by the lifting procedure in Proposition~\ref{prop:lifting}. If \(\cE^{\uparrow}=\binom{[n]}{k}\), then \(\cE=\binom{[n]}{k}\).
\end{lemma}

\begin{proof}[Proof of Lemma~\ref{lem:no-lift-to-full-k}]
It is enough to show that no set was lifted. Suppose, to the contrary, that at least one set was lifted. Look at the last lifting step. Let \(\cF_{0}\) be the family just before this step, and let \(\cF_{1}\) be the family just after it. Since this is the last lifting step and the final family is \(\binom{[n]}{k}\), we have \(\cF_{1}=\binom{[n]}{k}\).

In this last step the procedure replaces all members of the current minimum level, say level \(r\), by distinct \((r+1)\)-supersets. Since \(\cF_{1}\) consists only of \(k\)-sets, necessarily \(r+1=k\). Thus the last step replaces a nonempty family \(\cL\subseteq\binom{[n]}{k-1}\) by a family \(\cM\subseteq\binom{[n]}{k}\) with \(\abs{\cM}=\abs{\cL}\), where each inserted set contains the set it replaces.

Let
\[
\nabla\cL=\left\{K\in\binom{[n]}{k}:A\subseteq K\textup{ for some }A\in\cL\right\}.
\]
No member of \(\nabla\cL\) belongs to \(\cF_{0}\): otherwise it would contain some \(A\in\cL\), contradicting that \(\cF_{0}\) is an antichain. But \(\cF_{1}\) is the whole \(k\)-th level, so every set in \(\nabla\cL\) must be inserted during this last step. Therefore \(\nabla\cL\subseteq\cM\), and hence \(\abs{\nabla\cL}\le\abs{\cM}=\abs{\cL}\).

Now count pairs \(A\subseteq K\) with \(A\in\cL\) and \(K\in\nabla\cL\). Each \(A\in\cL\) is contained in \(n-k+1\) \(k\)-sets, while each \(K\in\nabla\cL\) contains at most \(k\) members of \(\cL\). Hence
\[
(n-k+1)\abs{\cL}\le k\cdot\abs{\nabla\cL}\le k\cdot\abs{\cL}.
\]
This is impossible because \(n>2k\). Hence no set was lifted, so \(\cE^{\uparrow}=\cE\).
\end{proof}

\begin{proof}[Completion of the proof of Theorem~\ref{thm:main}]
Let \(n>2k\). The upper bound was proved in the previous subsection, so suppose that \(\abs{\cF}=\binom{n}{k}\). We first move \(\cF\) into the middle band. Apply the lifting procedure in Proposition~\ref{prop:lifting} to \(\cF\), and call the resulting family \(\cF_{1}\). Thus every member of \(\cF_{1}\) has size at least \(k\). Next apply the same procedure to \(\cF_{1}^{c}\), call the resulting family \(\cR\), and set \(\cG=\cR^{c}\).

Complementation preserves the \([k]\)-differencing property, since \(\abs{A^{c}\setminus B^{c}}=\abs{B\setminus A}\). The first lifting gives \(\abs{A}\ge k\) for every \(A\in\cF_{1}\). During the second lifting, every changed sequence starts from a member \(R\in\cF_{1}^{c}\) with \(\abs{R}<k\) and ends on level \(k\) by Claim~\ref{clm:lifting-trajectory}. After taking complements, the corresponding set starts above level \(n-k\) and ends on level \(n-k\). Every unchanged member of \(\cF_{1}^{c}\) already has size at least \(k\), so its complement has size at most \(n-k\). It follows that \(\cG\) is \([k]\)-differencing, \(\abs{\cG}=\binom{n}{k}\), and \(k\le\abs{A}\le n-k\) for every \(A\in\cG\). Proposition~\ref{prop:middle-band-equality} now gives
\[
\cG=\binom{[n]}{k}\text{ or }\cG=\binom{[n]}{n-k}.
\]

Suppose first that \(\cG=\binom{[n]}{k}\). The second lifting made no changes: otherwise Claim~\ref{clm:lifting-trajectory} would give a \(k\)-set in \(\cR\), whose complement would be an \((n-k)\)-set in \(\cG\), impossible because \(n-k>k\). Hence \(\cF_{1}=\cG=\binom{[n]}{k}\). Lemma~\ref{lem:no-lift-to-full-k}, applied to the first lifting, gives \(\cF=\binom{[n]}{k}\).

Finally, suppose that \(\cG=\binom{[n]}{n-k}\). Then \(\cR=\cG^{c}=\binom{[n]}{k}\). Lemma~\ref{lem:no-lift-to-full-k}, applied to the second lifting, gives \(\cF_{1}^{c}=\binom{[n]}{k}\), and hence \(\cF_{1}=\binom{[n]}{n-k}\). The first lifting made no changes, since Claim~\ref{clm:lifting-trajectory} would otherwise put a \(k\)-set in \(\cF_{1}\), whereas every member of \(\cF_{1}\) has size \(n-k>k\). Therefore \(\cF=\binom{[n]}{n-k}\).
\end{proof}

\bibliographystyle{abbrv}
\bibliography{references}
\end{document}